\documentclass[11pt]{amsart}
\usepackage{amsmath,amsthm,verbatim,amssymb,amsfonts,amscd,  graphicx}
\usepackage[normalem]{ulem}
\usepackage{xcolor}
\usepackage{graphics}

\usepackage{euscript, enumerate}
\usepackage{url,hyperref}

\newcommand{\p}{\partial}

\newcommand{\e}{\varepsilon}

\newcommand{\eps}{\varepsilon}

\newcommand{\be}{\begin{equation}}
\newcommand{\ba}{\begin{aligned}}
\newcommand{\bee}{\begin{equation*}}
\newcommand{\ee}{\end{equation}}
\newcommand{\ea}{\end{aligned}}
\newcommand{\eee}{\end{equation*}}
\newcommand{\bea}{\begin{equation} \begin{aligned} }
\newcommand{\eea}{\end{aligned}\end{equation} }

\theoremstyle{plain}
\newtheorem{theorem}{Theorem}[section]

\newtheorem{lemma}[theorem]{Lemma}
\newtheorem{proposition}[theorem]{Proposition}

\newtheorem{claim}{Claim}[section]

\theoremstyle{remark}

\newtheorem{remark}[theorem]{Remark}

\theoremstyle{definition}
\newtheorem{definition}[theorem]{Definition}

\numberwithin{equation}{section}

\begin{document}
\title{Hitting estimates on Einstein manifolds and applications}
\author{Beomjun Choi and Robert Haslhofer}

\begin{abstract} We generalize the Benjamini-Pemantle-Peres estimate relating hitting probability and Martin capacity to the setting of manifolds with Ricci curvature bounded below. As applications we obtain: (1) a sharp estimate for the probability that Brownian motion comes close to the high curvature part of a Ricci-flat manifold, (2) a proof of an unpublished theorem of Naber that every noncollapsed limit of Ricci-flat manifolds is a weak solution of the Einstein equations, (3) an effective intersection estimate for two independent Brownian motions on manifolds with non-negative Ricci curvature and positive asymptotic volume ratio. We also obtain generalizations of (1) and (2) for the manifolds with two-sided Ricci bounds and Einstein manifolds with nonzero Einstein constant.
\end{abstract}
\maketitle

\section{Introduction}
By a classical theorem of Kakutani \cite{Kakutani}, Brownian motion in Euclidean space of dimension $n\geq 3$ hits a compact set with positive probability if and only if the set has positive Newtonian capacity. More recently, Benjamini-Pemantle-Peres \cite{BPP} discovered a sharp two-sided estimate for the hitting probability. One of their key insights was that in order to get the right dependence on the initial point $x_0$, one should work with the capacity associated to the Martin kernel
\begin{equation}
K(x,y)=\frac{|x-x_0|^{n-2}}{|x-y|^{n-2}}\, .
\end{equation}
Our first main result (see Theorem \ref{thm-hitting} below) generalizes the Benjamini-Pemantle-Peres estimate to the setting of Riemannian manifolds $(M^n,g)$ with non-negative Ricci-curvature, i.e.
\begin{equation}
\text{Ric} \ge 0.
\end{equation}
Let $p_t(x,y)$ be the heat kernel on $M$. By Cheeger-Yau \cite{CheegerYau} it satisfies the lower bound
\bea \label{eq_CY} p_t(x,y) \ge \frac{1}{(4\pi t)^{n/2}} e^{-\frac{d(x,y)^2}{4 t}},\eea
and by Li-Yau \cite{LiYau} for each $\gamma>4$ there is a $C_\gamma=C_\gamma(n)<\infty$ such that 
\bea \label{eq_LY}p_t(x,y)\le   \frac{C_\gamma\omega_n}{(4\pi)^{n/2} \text{Vol}(B(x,{\sqrt t}))} e^{-\frac{d(x,y)^2}{\gamma t}} .\eea
Let $X_t$ be Brownian motion on $M$ starting at $x_0\in M$ (for background see e.g. \cite{Hsu_book}). In terms of the heat kernel, $X_t$ is uniquely characterized by the formula
\begin{equation}\label{char_wiener}
\mathbb{P}_{x_0}[X_{t_1}\in U_1,\ldots ,X_{t_k}\in U_k]= \int_{U_1\times\ldots\times U_k} p_{t_1}(x_0,y_1) \cdots p_{t_k-t_{k-1}}(y_{k-1},y_k) dV(y_1)\ldots dV(y_k).
\end{equation}
For any Borel set $A\subseteq M$ the Martin capacity is defined by
\bea \label{eq-martincap} \text{Cap}_K(A):= \left[\inf_{\nu(A)=1}  \iint_{A\times A} K(x,y)\, d\nu (x) d\nu(y)  \right]^{-1} ,\eea
where $K$ denotes the Martin kernel
\bea\label{eq_martin}
K(x,y)=\begin{cases} \begin{aligned} &\left(\frac{d(x_0,y)}{d(x,y)}\right)^{n-2} & \text{ if } x\neq y \\ &\qquad \infty & \text{ if } x=y\end{aligned}\end{cases}.
\eea
Using these notions, we can now state our hitting estimate:

\begin{theorem} \label{thm-hitting}Let $(M^n,g)$ be a complete Riemannian manifold of dimension $n\geq 3$, with non-negative Ricci-curvature.  Then for all compact sets $A\subseteq M^n $ and all $T\leq\infty$ we have the hitting estimate
 \bea\label{eq-hittingprob} \frac{1}{2\Lambda}  \mathrm{Cap}_{K}(A)\le \mathbb{P}_{x_0}[ X_t \in  A \textrm{ for some } { 0<t< T} ] \le  \Lambda \mathrm{Cap}_{K}(A), \eea 
where
\bea \label{eq-C2}\Lambda=    \frac{ \omega_n}{v} \left ( \frac{ \gamma}{4}\right )^{\frac n 2 -1 }  C_\gamma  \frac{  \int_{0}^\infty   \xi^{\frac{n}{2} -2} e^{-\xi} d\xi } { \int_{ \frac{\mathrm{diam}(A\cup \{x_0\})^2}{4T}}^\infty  \xi^{\frac{n}{2} -2} e^{-\xi} d\xi },  \eea
with \bea \label{eq-v'} v = \inf_{y\in A \cup \{x_0\}} \frac{\mathrm{Vol}(B(y,\sqrt{2T}))}{(2T)^{n/2}}.  \eea 
 \end{theorem}

\begin{remark}
\begin{enumerate}[(i)]
\item If $M=\mathbb{R}^n$ is flat Euclidean space, then we can choose $\gamma=4$, $C_{\gamma}=1, v=\omega_n$ and $T=\infty$, and hence recover the sharp estimate
\begin{equation*}
 \frac{1}{2}  \mathrm{Cap}_{K}(A)\le \mathbb{P}_{x_0}[ B_t \in  A \textrm{ for some }0<t<\infty] \le  \mathrm{Cap}_{K}(A)
\end{equation*}
from Benjamini-Pemantle-Peres \cite{BPP}.
\item In practice, one usually chooses $T\gg\mathrm{diam}(A\cup \{x_0\})^2$ so that $\Lambda\approx   \frac{ \omega_n}{v} \left ( \frac{ \gamma}{4}\right )^{\frac n 2 -1 }  C_\gamma$. In particular, if $(M,g)$ has nonvanishing asymptotic volume ratio one can simply choose $T=\infty$.
\item Assuming $\text{Vol}(B(x_0,r_0)) \geq v_0r_0^n$, by Bishop-Gromov volume comparison we of course have
\begin{equation*}
v\geq \left( \frac{r_0}{\max\{\sqrt{2T},r_0+\mathrm{diam}(A\cup \{x_0\})\} } \right)^n v_0\, .
\end{equation*}
\end{enumerate}
\end{remark}

\bigskip

One of the main features of Theorem \ref{thm-hitting}, in contrast to prior hitting estimates on manifolds in the literature (see e.g. \cite{ChavelFeldman,GrigSal}), is that our estimate is universal in the sense that it only depends on a lower bound for the noncollapsing constant. Let us now discuss three applications.\\

Our first application concerns the question: How likely is it that Brownian motion comes close to the almost singular part of a Ricci-flat manifold?
 To discuss this, recall that for $x\in M$ the regularity scale is defined by
 \bea\label{eq-regularityscale}
 \mathrm{reg}(x):=\sup\big\{ r\leq 1\, : \, \sup_{B(x,r)} |\mathrm{Rm}|\leq r^{-2}\big\}.
 \eea
We consider the $\eps$-singular set 
 \bea\label{eq-esingular}
 S_\eps:= \{ x \in M\, : \, \mathrm{reg}(x)\leq \eps \},
 \eea
and its $\eps$-tubular neighborhood
 \bea
 N_\eps(S_\eps)=\{ x\in M \, : \, d(x,S_\eps)\leq \eps\}.
 \eea
Similarly, we consider the Wiener sausage
\bea
N_\eps(X[0,T])= \{ x\in M \, : \, d(x, X_t)\leq \eps \textrm{ for some } 0\leq t\leq T\}\, .
\eea
One can think of the Wiener sausage as the path swept out by a Brownian particle of size $\eps$. The following theorem estimates the probability that this particle hits the almost singular part $N_\eps(S_\eps)$:

\begin{theorem}\label{thm_app1}
Let $(M^n,g)$ be a complete Riemannian Ricci-flat manifold of dimension $n\geq 3$.  Then for all $0<\eps\leq 1$, $1\leq T<\infty$ and all $x_0\in M$ with $d(x_0,S_\eps) \ge 3\eps$ we have the estimate
\begin{equation}
\mathbb{P}_{x_0}\left[N_\eps(X[0,T])\cap N_\eps(S_\eps) \cap B(x_0,\sqrt T ) \neq \emptyset  \right] \le  \frac{ C\eps^2}{d(x_0,S_\eps)^{2}},
\end{equation}
where $C=C(n,v)<\infty$ only depends on the dimension, and a lower bound $v$ for $\frac{\mathrm{Vol}(B(x_0,\sqrt T))}{T^{n/2}}$.
\end{theorem}

The estimate in Theorem \ref{thm_app1} is sharp as illustrated by the following proposition:

\begin{proposition}\label{prop_EH}
If $(M,g)=(T^\ast S^2\times\mathbb{R}^{n-4}, g_\eps+\delta)$, where $g_\eps$ is the Eguchi-Hanson metric with central two-sphere of size $\eps$, then for all $x_0\in M$ with $3\eps\leq d(x_0,S_\eps)\leq \sqrt{T}/3$ we have the lower bound
\begin{equation}
\mathbb{P}_{x_0}\left[X[0,T]\cap S_\eps \cap B(x_0,\sqrt T ) \neq \emptyset  \right] \geq  \frac{ c\eps^2}{d(x_0,S_\eps)^{2}}.
\end{equation}
\end{proposition}

\bigskip

Our second application concerns weak solutions of the Einstein equations. To discuss this, recall that Naber \cite{Naber13} discovered a characterization of Ricci-flat metrics via a sharp gradient estimate on path space. Specifically, he proved that a Riemannian manifold $(M^n,g)$ solves the Einstein equations
\bea
\textrm{Ric}=0
\eea
if and only if the gradient estimate
\begin{equation}\label{grad_path}
|\nabla_x \mathbb{E}_x [F]| \leq \mathbb{E}_x [|\nabla^\parallel F|]
\end{equation}
holds for all test functions $F:PM\to\mathbb{R}$ on path space and almost every $x\in M$. Here, $\mathbb{E}_x$ denotes the expectation with respect to the Wiener measure of Brownian motion starting at $x\in M$, and $\nabla^\parallel F(X)\in T_{x}M$ denotes the parallel gradient, which for cylinder functions $F(X)=f(X_{t_1},\ldots,X_{t_k})$ is defined by
\bea
\nabla^\parallel F(X) = \sum_{i=1}^k P_{t_i}(X)\nabla^{(i)}f(X_{t_1},\ldots,X_{t_k})\, ,
\eea
where $P_t(X):T_{X_t}M\to T_xM$ denotes stochastic parallel transport. Motivated by this, Naber then introduced a notion of weak solutions of the Einstein equations for metric measure spaces. To describe this, let us first impose the convention that all metric measure spaces $(M,d,m)$ are assumed to be locally compact, complete, length spaces such that $m$ is a locally finite, $\sigma$-finite Borel measure with $\mathrm{spt}(m)=M$. Now, recall that on any metric measure space $(M,d,m)$ one has the Cheeger-energy
\begin{equation}
E[u] =\tfrac{1}{2}\int_M |\nabla u|_\ast^2\, dm,
\end{equation}
where $|\nabla u|_\ast$ denotes the minimal relaxed gradient \cite{AGS1}. If $E$ satisfies the parallelogram identity, then $M$ is called Riemannian. Any Riemannian metric measure space has a linear heat flow and can be equipped with a Wiener measure of Brownian motion characterized by \eqref{char_wiener},
see \cite{Fukushima}.
\begin{definition}[{c.f. Naber \cite{Naber13}}] \label{def-weakflat}
A Riemannian metric measure space $(M,d,m)$ is called a weak solution of the Einstein equations if the gradient estimate
\begin{equation}\label{grad_path2}
|\nabla_x \mathbb{E}_x [F]| \leq \mathbb{E}_x [|\nabla^\parallel F|]
\end{equation}
holds for all test functions $F:PM\to\mathbb{R}$ on path space and $m$-almost every $x\in M$, where $\nabla_x$ is the local Lipschitz slope and the parallel gradient is the one constructed in \cite[Section 14]{Naber13}.
\end{definition}

Using Theorem \ref{thm_app1} we give a proof of the following unpublished result of Naber:\footnote{RH thanks Aaron Naber for pointing out in 2014 that this follows from the solution of the codimension 4 conjecture by Cheeger-Naber \cite{CheegerNaber}, which in particular implies that the singular set has zero capacity.}

\begin{theorem}[Naber]\label{thm_ricci_limit} 
Every noncollapsed limit of Ricci-flat manifolds is a weak solution of the Einstein equations.
\end{theorem}

For comparison, recall that by the theory of Lott-Villani \cite{LottVillani} and Sturm \cite{Sturm} lower bounds on Ricci-curvature are preserved under weak limits. While lower curvature bounds survive even in the collapsing case, for upper curvature bounds the noncollapsing condition is essential. 

\bigskip

For  simplicity, we have stated the results for the manifolds with non-negative Ricci curvature or Ricci-flat manifolds (or their limits). In fact, small modifications of those statements and proofs give analogous results for the manifolds with lower Ricci curvature bound or Einstein manifolds of nonzero Einstein constants (or their limits). We discuss these generalizations in Section \ref{sec-othercase}. 

\bigskip

Our final application concerns the question: How likely is it that two independent Brownian motions come close to each other? For Brownian motion in Euclidean space of dimension $n>4$ by work of Aizenman \cite{Aizenman} and Pemantle-Peres-Shapiro \cite{PPS} the probability that two independent Brownian particles of size $\eps$ starting apart hit each other is proportional to $\eps^{n-4}$. The following theorem generalizes this to manifolds with non-negative Ricci curvature with nonvanishing asymptotic volume ratio:

  \begin{theorem}\label{thm_eff_int} For any $n>4$ and $v>0$ there exists a constant $C=C(n,v)<\infty$ with the following significance. For any complete Riemannian manifold $(M,g)$ of dimension $n>4$ with $\mathrm{Ric}\geq 0$ and $\lim_{r\to \infty}\frac{\mathrm{Vol}(B(x,r))}{r^n}\geq v$, for any two points $x_0,\tilde x_0\in M$ with $d(x_0,\tilde x_0) > 5\eps$ we have
  \bea
C^{-1}  \frac{\eps^{n-4}}{d(x_0,\tilde x_0)^{n-4}} \le \mathbb{P}_{{x}_0,\tilde x_0}[N_\eps(X[0,\infty))\cap N_\eps(\tilde{X}[0,\infty)) \neq \emptyset ] \le C \frac{\eps^{n-4}}{d(x_0,\tilde x_0)^{n-4}}.\eea 
\end{theorem}
 
Again the estimate is universal in the sense that it only depends on a lower bound for the asymptotic volume ratio.\\

This article is organized as follows: In Section \ref{sec_hitting}, we prove the hitting estimate. To do so, we generalize the approach by Benjamini-Pemantle-Peres \cite{BPP} to the manifold setting, using in particular the heat kernel estimates from Cheeger-Yau and Li-Yau. The upper bound is derived via an entrance time decomposition, while the lower bound relies on a second moment estimate. In Section \ref{sec_app}, we prove the theorems from our three applications. To prove Theorem \ref{thm_app1}, we have to estimate the Martin capacity of the almost singular set $N_\eps(S_\eps)$. A key ingredient for this is the work of Jiang-Naber \cite{JiangNaber}. To prove Theorem \ref{thm_ricci_limit}, we show that almost every path of Brownian motion starting at a regular point stays in the regular part, and then conclude as in \cite{Naber13}. To prove Theorem \ref{thm_eff_int}, we have to estimate the expected Martin capacity of the Wiener sausage. We do this via suitable dyadic decompositions and repeated applications of Theorem \ref{thm-hitting}. Finally, in Section \ref{sec-othercase}, we prove variants of Theorem \ref{thm-hitting} and Theorem \ref{thm_app1} for the manifolds with lower Ricci bound and two-sided Ricci bounds, respectively. Moreover, as a generalization of Theorem \ref{thm_ricci_limit}, we prove every noncollapsed limit of Einstein manifolds with Ricci bounds has Ricci curvature bounds in generalized sense of \cite{Naber13}.  

\bigskip

\noindent\textbf{Acknowledgements.}
The second author has been supported by an NSERC Discovery Grant and a Sloan Research Fellowship.\\

\bigskip

\section{Proof of the hitting estimate}\label{sec_hitting}
 
\begin{proof}[Proof of Theorem \ref{thm-hitting}]

To prove the upper bound, we consider an entrance time decomposition. To this end, let
\begin{equation}
 \tau := \inf \{t>0\,:\, X_t \in A \} \in [0,\infty]
\end{equation}
be the first hitting time of $A$.  Note that $\tau \wedge T$ is a stopping time. Let $\mu$ be the distribution of $X_{\tau\wedge T}$, i.e.
\be  \mu (A') = \mathbb{P}_{x_0} [X_{\tau\wedge T}  \in A'].\ee 
Observe that
\be\label{property_mu}\mathbb{P}_{x_0}[ X_t \in  A \textrm{ for some } { 0<t< T} ]=\mu(A). \ee  
Now, for any Borel set $A'\subseteq A$ we consider the expected occupancy time
\be
\mathbb{E}_{x_0}\left[ \int_0^{2T} 1_{\{X_{t}\in A'\} } dt \right] =  \int_0^{2T}\int_{A'} \rho_t(x_0,y) \,  dV(y)\, dt \, .
\ee
By the strong Markov property, $X'_t = X_{\tau\wedge T +t}$ is a Brownian motion starting from $X_{\tau \wedge T}$. Using this and the lower heat kernel bound \eqref{eq_CY} from Cheeger-Yau we can estimate
\bea\label{eq-a2}
\mathbb{E}_{x_0}\left[ \int_0^{2T} 1_{\{X_{t}\in A'\} } dt \right]&\geq \mathbb{E}_{x_0}\left[ \int_{\tau\wedge T} ^{2T} 1_{X_{t}\in A' } dt \right]\\
 &= \int_M \mathbb{E}_x\left[ \int_0 ^{2T-(\tau\wedge T)} 1_{X'_{t}\in A' } dt \right] d\mu(x)\\ 
&=  \int_M \int_{A'} \int _0^{(2T-\tau) \vee  T}\rho_t(x,y)\,dt \, dV(y)\, d\mu(x)  \\
&\ge\int_A \int_{A'} \int _0^{  T} \frac{1}{(4\pi t)^{n/2}} e^{-\frac{ d(x,y)^2}{4t}}\, dt \, dV(y)\, d\mu(x). 
\eea 
On the other hand, by the upper heat kernel bound \eqref{eq_LY} from Li-Yau we can estimate
 \bea  \mathbb{E}_{x_0}\left[ \int_0^{2T} 1_{\{X_{t}\in A'\} } dt \right]\le \int_{A'} \int_{0}^{2T}  \frac{ \omega_n}{v}\frac{C_\gamma }{(4\pi t)^{n/2}} e^{-\frac{d(x_0,y)^2}{\gamma t}}\, dt \, dV(y)\, .
 \eea
Since $A'$ was arbitrary, combining the above inequalities we infer that
\bea
\int_A  \int _0^{  T} \frac{1}{(4\pi t)^{n/2}} e^{-\frac{ d(x,y)^2}{4t}}\, dt \, d\mu(x)\leq \int_{0}^{2T}  \frac{ \omega_n}{v}\frac{C_\gamma }{(4\pi t)^{n/2}} e^{-\frac{d(x_0,y)^2}{\gamma t}}\, dt\, 
\eea
for $dV$-almost every $y\in A$. Via substitution, we can rewrite this in the form
\bea\label{cf_eq28}
\int_A  \int_{\frac{d(x,y)^2}{4T}}^\infty \frac{\xi^{\frac{n}{2}-2}}{d(x,y)^{n-2}} e^{- \xi}\, d\xi \, d\mu(x)\leq \frac{  \omega_n}{v} \left ( \frac{ \gamma}{4}\right )^{\frac n 2 -1 } C_\gamma \int_{\frac{d(x,y)^2}{2\gamma T}}^\infty  \frac{\xi^{\frac{n}{2}-2}}{d(x_0,y)^{n-2}} e^{- \xi}\, d\xi \, .
\eea
Recalling the definition of the Martin kernel \eqref{eq_martin}, this shows that
\bea\label{eq-bda}
\int_A K(x,y) \, d\mu(x)\leq \frac{  \omega_n}{v} \left ( \frac{ \gamma}{4}\right )^{\frac n 2 -1 }C_\gamma \left[\frac{  \int_{0}^\infty   \xi^{\frac{n}{2} -2} e^{-\xi} d\xi } { \int_{ \frac{\mathrm{diam}(A)^2}{4T}}^\infty  \xi^{\frac{n}{2} -2} e^{-\xi} d\xi } \right]\, .
\eea
At first, this only follows for $dV$-almost every $y\in A$, but considering the regularized kernel $K_\eps(x,y)=\left(\frac{d(x_0,y)}{d(x,y)+\eps}\right)^{n-2}$ and sending $\eps\to 0$ we see that \eqref{eq-bda} in fact holds for every $y\in A$.
Finally, integrating \eqref{eq-bda}  with respect to $d\mu(y)$ we conclude that
\bea
\int_{A\times A} K(x,y) \, d\mu(x)\, d\mu(y) \leq \Lambda \mu(A).
\eea
Recalling \eqref{property_mu}, this proves the upper bound
\bea
\mathbb{P}_{x_0}[ X_t \in  A \textrm{ for some } { 0<t< T} ]\leq \Lambda \mathrm{Cap}_K(A).
\eea
\bigskip

To prove the lower bound, we use a second moment estimate. By general properties of $\mathrm{Cap}_K$, we can assume without loss of generality that $x_0\notin A$.\footnote{Otherwise, consider $A\setminus B(0,1/i)$ and use that $\lim_{i\to\infty} \mathrm{Cap}_K(A\setminus B(x_0,1/i))=\mathrm{Cap}_K(A)$.} Given any probability measure $\nu$ on $A$, let us consider the random variable 
\bea Z_\e:= \int_A \frac{\int_0^{T}1_{\{X_t \in B(y,\e) \}}dt}{\mathbb{E}_{x_0} [\int_0^{T}1_{\{X_t \in B(y,\e)\}}dt ]}\, d\nu(y). \eea
 By construction, we have $\mathbb{E}_{x_0} [Z_\e] =1$. The random variable $Z_\e$ measures a weighted occupancy time of the $\e$-tubular neighborhood of $N_\e(A)$, in particular
 \bea\label{inpart_est} \mathbb{P}_{x_0} [X_t \in N_\e(A)\text{ for some } 0 < t< T] & \ge  \mathbb{P}_{x_0}[ Z_\e >0]\\
& \ge \frac{1}{\mathbb{E}_{x_0}[Z_\e^2]},
 \eea
 where we used the Cauchy-Schwarz inequality in the last step. With the goal of deriving an upper bound for $\mathbb{E}_{x_0}[Z_\e^2]$, let us define the functions
 \bea h_\e (x,y) :=   \mathbb{E}_{x}\left[ \int_0^T 1_{\{X_t \in B(y,\e)\}} \, dt\right], \eea
and
\bea h^*_\e (x,y):= \sup_{z\in B(x,\e)} h_\e(z,y). \eea 
Using the strong Markov property we can now estimate
\bea\label{eq_sec_mom_est} \mathbb{E}_{x_0}[Z^2_\e ] &=2\mathbb{E}_{x_0}\left[\int_A \int_A  \int_0^{T}  \int_s^{T}\frac{1_{\{X_t \in B(y,\e)\}\cap\{X_s\in B(x,\e) \} }}{h_\e(x_0,y)h_\e(x_0,x) } \,dt \,ds\right] d\nu(x)\, d\nu(y)  \\
&\le  2\mathbb{E}_{x_0}\left[ \int_A \int_A  \int_0^{T} 1_{\{X_s\in B(x,\e) \} }\frac{ h^*_\e(x,y)}{h_\e(x_0,y)h_\e(x_0,x) }   \right] ds\, d\nu(x)\, d\nu(y)\\
&= 2   \int_A \int_A \frac{ h^*_\e(x,y)}{h_\e(x_0,y)}  \, d\nu(x)\, d\nu(y) .\eea 
To proceed, we need the following claim.
\begin{claim}\label{lemma_claim} We have \bea\label{eq-claim2} \lim_{\e\searrow 0}    \int_A \int_A \frac{ h^*_\e(x,y)}{h_\e(x_0,y)} \, d\nu(x)\, d\nu(y) = \int_A \int_A \frac{\int_0^{T} \rho_t(x,y) \, dt }{\int_0^{T} \rho_t(x_0,y) \, dt} \,d\nu(x)\, d\nu(y).\eea \end{claim}
\begin{proof}[{Proof of Claim \ref{lemma_claim}}]
First observe that
 \[ \lim_{\e\searrow 0} \frac{   h_\e(x_0,y)}{\text{Vol}(B(y,\e))} = \int_0^{T} \rho_t(x_0,y)\, dt  \quad \textrm{and} \quad \lim_{\e\searrow 0}\frac{ h^*_\e(x,y)}{\text{Vol}(B(y,\e))}  = \int_0^{T} \rho_t(x,y)\, dt.\]  
Hence, we have the pointwise convergence
\bea
\lim_{\e\searrow 0} \frac{ h^*_\e(x,y)}{h_\e(x_0,y)}=\frac{\int_0^{T} \rho_t(x,y)\, dt }{\int_0^{T} \rho_t(x_0,y)\, dt}\, .
\eea
To take the limit under the integral it suffices to find a constant $C<\infty$ such that 
\bea \label{eq-claimproof}\frac{h^*_\e(x,y)}{h_\e(x_0,y)} \le C \frac{\int_0^T \rho_t(x,y)\,dt}{\int_0^T \rho_t(x_0,y)\, dt}\eea
for $\e\leq \tfrac{1}{2}d(x_0,A)\wedge T^{1/2}$.  To this end, first observe that using the heat kernel estimates \eqref{eq_CY} and \eqref{eq_LY}  from Cheeger-Yau and Li-Yau, similarly as in \eqref{cf_eq28}, we obtain
\bea \label{green_eq}C^{-1} d(x,y)^{2-n} \le \int_0^T \rho_t(x,y)\, dt \le C d(x,y)^{2-n} \eea 
for all $x,y\in N_\eps(A)\cup\{x_0\}$, where $C$ depends on $n$, $v$, $T$ and $\textrm{diam}(A\cup\{x_0\})$. Therefore,  the ratio
\bea
\frac{   h_\e(x_0,y)}{\text{Vol}(B(y,\e))}= \frac{  1}{\text{Vol}(B(y,\e))}\int_{B(y,\e)} \int_0^T\rho_t(x_0,z)\, dt\, dV(z)
\eea
is uniformly bounded from above and below by positive constants. Thus, it suffices to show
\bea  \frac{h^*_\e (x,y)}{\mathrm{Vol}(B(y,\e))} \le C \int_0^T \rho_t(x,y)\, dt.\eea
This clearly holds if $d(x,y) \geq 3\e$. On the other hand, if $d(x,y) < 3\e$, then using the coarea formula and Bishop-Gromov volume comparison, for $z$ with $d(x,z)<\e$ we compute
\bea\int_{B(y,\e)} \int_0^{T}\rho_t(z,\zeta ) \, dt \, dV(\zeta) &\le \int_{B(z,5\e)}  \int_0^{T}\rho_t(z,\zeta) \, dt \, dV(\zeta) \\
 & \le C \int_0^{5\e} \int_{\p B(z,r)}  d(z,\zeta)^{2-n}  d\mathcal{H}^{n-1}(\zeta) \, dr   \\
  &\le C \text{Vol}(B(y,\e))\int_0^T \rho_t(x,y) \, dt.\eea 
This finishes the proof of the claim.
\end{proof}
Continuing the proof of the theorem, by estimate \eqref{eq_sec_mom_est} and Claim \ref{lemma_claim} we have
\bea
\limsup_{\e\searrow 0}\mathbb{E}_{x_0}[Z^2_\e ]\leq2 \int_A\int_A  \frac{\int_0^{T} \rho_t(x,y) \, dt }{\int_0^{T} \rho_t(x_0,y) \, dt} \,d\nu(x)\, d\nu(y).
\eea
Since 
\bea  \frac{\int_0^{T} \rho_t(x,y) \, dt }{\int_0^{T} \rho_t(x_0,y) \, dt} 
&\le   \frac{ \omega_n}{v} \left ( \frac{ \gamma}{4}\right )^{\frac n 2 -1 }  C_\gamma  \left[ \frac{ \int_{{0}}^ \infty \xi^{\frac n2-2}e^{-\xi}\, d\xi }{ \int_{\frac{\sup_{y\in A} d(x_0,y)^2}{4T} }^ \infty \xi^{\frac n2-2}e^{-\xi}\, d\xi } \right]K(x,y),\eea
we infer that
\bea \limsup_{\e\searrow 0} \mathbb{E}_{x_0}[Z^2_\e]  \le 2 \Lambda  \int_A\int_A K(x,y) \, d\nu(x) \, d\nu(y) .\eea 
Together with \eqref{inpart_est} and the continuity of Brownian paths this shows that
 \bea&\mathbb{P}_{x_0}[ X_t \in A \text{ for some } 0<t< T] \ge \frac{1}{2\Lambda }\left [  \int_A \int_A K(x,y) \, d\nu(x) \, d\nu(y) \right]^{-1}.\eea
 Since the probability measure $\nu$ was arbitrary, this proves the lower bound
  \bea&\mathbb{P}_{x_0}[ X_t \in A \text{ for some } 0<t<T] \ge \frac{1}{2\Lambda }\mathrm{Cap}_K(A),\eea
  which concludes the proof of the theorem.
\end{proof}

 \bigskip

\section{Applications}\label{sec_app}

\subsection{Probability to come close to almost singular part}\label{sec_appl1}

In this section, we prove Theorem \ref{thm_app1}.  Recall that for any metric space $(X,d)$, the $\alpha$-dimensional Hausdorff measure is defined by
\bea
  \mathcal{H}^\alpha (A)= \lim_{\e\searrow 0} \mathcal{H}^\alpha_\e (A),
 \eea 
  where 
 \bea \mathcal{H}^\alpha_\e (A) := \inf \left \{ \sum_i \mathrm{diam}(A_i)^\alpha \,:\,A \subseteq \bigcup_i A_i \text{ and } \mathrm{diam}(A_i)\le \e  \right\} .\eea 
We will need the following lemma relating capacity and Hausdorff measure.

\begin{lemma} \label{lem-capacity}Let $(X,d)$ be a metric space and $A\subseteq X$ be a Borel set. Then
\bea \label{eq-hausdorffcap}\mathcal{H}_{\e}^{\alpha} (A) \ge  \frac{ \mu(A)^2}{\iint_{\{d(x,y)\le \e\}} d(x,y)^{-\alpha } d\mu(x) d\mu(y) }\eea for any $\alpha>0$ and any nonzero finite (Borel) measure $\mu$ on A.
\end{lemma} 

\begin{proof}[{Proof of Lemma \ref{lem-capacity}}] If $A_i$ is a countable disjoint covering of $A$ whose diameters are bounded by $\e$, then
\bea \mu(A)^2 & \le  \left[\sum _i\mathrm{diam}(A_i)^{\alpha/2} \mathrm{diam}(A_i)^{-\alpha/2}   \mu(A_i)\right]^2 \\
&\le  \sum_i \mathrm{diam}(A_i)^\alpha  \sum_i \mathrm{diam}(A_i)^{-\alpha} \mu(A_i)^2  \\
& \le \sum_i \mathrm{diam}(A_i)^\alpha  \sum_i \int_{A_i}\int_{A_i}d(x,y)^{-\alpha} d\mu(x) d\mu(y)  \\
&\le \sum_i \mathrm{diam}(A_i)^\alpha \iint_{ \{d(x,y) \le \e \} } d(x,y)^{-\alpha} d\mu(x) d\mu(y).\eea
This proves the lemma.
\end{proof}

We can now prove the main result of this section:

\begin{proof}[Proof of Theorem \ref{thm_app1}]

First observe that
 \bea
 N_\eps(X[0,T])\cap N_\eps(S_\eps)\neq \emptyset\quad \Rightarrow \quad X_t\in N_{2\eps}(S_\eps) \textrm{ for some } 0<t<T.
 \eea
  Hence, using Theorem \ref{thm-hitting} we can estimate
 \bea
\mathbb{P}_{x_0}[ N_\eps(X[0,T])\cap N_\eps(S_\eps)\cap B(x_0,\sqrt{T})\neq \emptyset] \leq \Lambda \textrm{Cap}_K(N_{2\eps}(S_\eps)\cap B(x_0,\sqrt{T})).
 \eea
To proceed, let us write $\textrm{Cap}_N$ for the capacity associated to the Newtonian kernel
\bea
N(x,y)=\frac{1}{d(x,y)^{n-2}}\, .
\eea
In general, the Martin capacity and the Newtonian capacity can be compared via 
 \bea\label{eqMK}\frac{\mathrm{Cap}_N(A)}{\sup_{x\in A} d(x_0,x)^{n-2}} \le \mathrm{Cap}_K(A) \le\frac{\mathrm{Cap}_N(A)}{\inf_{x\in A} d(x_0,x)^{n-2}} \eea 
To use this efficiently, let us chop up our set into dyadic annuli. Namely, set $r_i:= 2^i d(x_0, N_{2\eps}(S_\eps))$, let $k$ be the smallest integer such that $r_k\ge \sqrt{T}$, and estimate
  \bea \mathrm{Cap}_K(N_{2\eps}(S_\eps) \cap B(x_0,\sqrt{T}))& \le \sum_{i=1}^{k} \mathrm{Cap}_K(N_{2\eps}(S_\eps) \cap (B(x_0, r_i) \setminus B(x_0,r_{i-1}))) \\
  &\le  \sum_{i=1}^{k} \frac{1}{r_{i-1}^{n-2}} \mathrm{Cap}_N(N_{2\eps}(S_\eps) \cap B(x_0, r_i)). \\
  \eea
 Now by the deep work of Jiang-Naber \cite[Theorem 1.8]{JiangNaber}, we have
 \bea\label{JiangNaber_est}
 \mathrm{Vol}{(N_{2\e}(S_{\e}) \cap B(x_0,r_i))} \le C \e ^4r_i^{n-4},
 \eea
 where $C$ only depends on the dimension $n$, and a lower bound $v$ for $\mathrm{Vol}(B(x_0,\sqrt{T}))/{T^{n/2}}$.   Hence, we can find $J_i\leq C\eps^{4-n}r_i^{n-4}$ points $x_{i,j}\in N_{2\eps}(S_\eps)\cap B(x_0,r_i)$, such that
 \bea
 N_{2\eps}(S_\eps)\cap B(x_0,r_i) \subseteq \bigcup_{j=1}^{J_i} B(x_{i,j},\eps)\, .
 \eea
This yields
 \bea
 \mathrm{Cap}_N(N_{2\eps}(S_\eps) \cap B(x_0, r_i))\leq J_i\max_j \text{Cap}_N(B(x_{i,j},\eps))\leq C \eps^2 r_i^{n-4},
  \eea
where we also used that by Lemma \ref{lem-capacity} the Newtonian capacity of $\eps$-balls can be estimated by
 \bea\label{eq_cap_ball} \text{Cap}_N(B(x,\eps)) \le \mathcal{H}^{n-2}_\eps(B(x,\eps))\le C \eps^{n-2}.\eea
Putting things together we conclude that
  \bea \mathbb{P}_{x_0}[ N_\eps(X[0,T])\cap N_\eps(S_\eps)\cap B(x_0,\sqrt{T})\neq \emptyset] 
  \leq C \sum_{i=1}^k \frac{1}{r_{i-1}^{n-2}} \eps^2 r_i^{n-4}\leq C \frac{\eps^2}{r_0^2}.
   \eea
This proves the theorem.
\end{proof}

Finally, let us show that the estimate is sharp by considering the example of the Eguchi-Hanson metric \cite{EguchiHanson}.

\begin{proof}[{Proof of Proposition \ref{prop_EH}}] Let $r_0:=d(x_0, S_\e)$. In view of the lower bound from Theorem \ref{thm-hitting} it is enough to show that
 \bea  \mathrm{Cap}_K (S_\e \cap B(x_0,3r_0) ) \ge  c \e^2 r_0^{-2} . \eea 
To begin with, by \eqref{eqMK} we have
  \bea \mathrm{Cap}_K(  S_\e\cap B(x_0,3r_0))   \ge (3r_0)^{2-n} \mathrm{Cap}_N(  S_\e\cap B(x_0,3r_0))  .\eea
  Next, observe that $S_\e\cap B(x_0,3r_0)$ contains a set of the form
  \bea
  Z:=B^4(p,\eps)\times B^{n-4}(q,2r_0)
  \eea
where $p$ is any point on the central two-sphere of the Eguchi-Hanson manifold, and $q\in \mathbb{R}^{n-4}$. Letting $\nu$ be the uniform probability measure on $Z$, an elementary computation shows that
\bea
\int_Z\int_Z \frac{d\nu(x)\, d\nu (y)}{d(x,y)^{n-2}}&= \frac{1}{\textrm{Vol}(Z)^2}\int_Z\int_Z \frac{dV(x)\, dV (y)}{d(x,y)^{n-2}}\leq \frac{C\eps^2}{\textrm{Vol}(Z)}\leq \frac{C}{\eps^2 r_0^{n-4}}.
\eea
Hence,
\bea
\mathrm{Cap}_N(Z)\geq C^{-1}\eps^2r_0^{n-4}.
\eea  
 Putting things together, this proves the proposition.
 \end{proof}

  \bigskip
  
\subsection{Weak solutions of the Einstein equations} In this section we prove that every noncollapsed limit of Ricci-flat manifolds is a weak solution of the Einstein equations.

\begin{proof}[{Proof of Theorem \ref{thm_ricci_limit}}]

Let $(M,d,m)$ be a noncollapsed limit of Ricci-flat manifolds.
The key is to establish the following claim.

\begin{claim}\label{claim_key}
For every regular point $x_0\in M$ almost every path of Brownian motion starting at $x_0$ stays in the regular part of $M$.
\end{claim}

\begin{proof}[Proof of the claim] Let $S\subseteq M$ be the singular set. Since a countable union of null events is a null event it is enough to show that for every $0<R<\infty$ and every $T<\infty$ we have
\bea
\mathbb{P}_{x_0}[ X_t \in  S \cap B(x_0,R) \textrm{ for some } { 0<t< T} ] = 0.
\eea
Let $S_\eps\subseteq M$ be the quantitative singular set. For any positive integer $N<\infty$ consider the partition $t_k=k T/N$, where $k=1,\ldots, N$.
By Lemma \ref{std_lemma} (see below) we can estimate
\begin{multline}
\mathbb{P}_{x_0}[ X_t \in  S_\eps \cap B(x_0,R) \textrm{ for some } { 0<t<T} ] \\
\leq \mathbb{P}_{x_0}[ X_{t_k} \in N_\eps(S_\eps) \cap B(x_0,R) \textrm{ for some } 1\leq k\leq N ] + N e^{-\frac{\eps^2N}{100}}.
\end{multline}
By assumption we have the pointed measured Gromov-Hausdorff convergence $(M_i,g_i,dV_i)\to (M,d,m)$. Using the local regularity theorem \cite{Anderson} and the convergence of the heat kernel \cite{Ding} we can estimate
\begin{multline}\label{eq-321}
\mathbb{P}_{x_0}[ X_{t_k} \in N_\eps(S_\eps) \cap B(x_0,R) \textrm{ for some } 1\leq k\leq N ]\\
 \leq \limsup_{i\to \infty}\mathbb{P}_{x_0^i}[ B_{t_k}^i \in N_{2\eps}(S_{2\eps}^i)\cap B(x_0^i,R+1) \textrm{ for some } 1\leq k\leq N ] \, .
\end{multline}
By Theorem \ref{thm_app1} this is bounded by $C\eps^2$ for $\eps$ small enough. Hence,
\bea
\mathbb{P}_{x_0}[ X_t \in  S_\eps \cap B(x_0,R) \textrm{ for some } { 0<t< T} ] \leq C\eps^2 + N e^{-\frac{\eps^2N}{100}}.
\eea
Sending $N\to \infty$ this yields
\bea
\mathbb{P}_{x_0}[ X_t \in  S_\eps \cap B(x_0,R) \textrm{ for some } { 0<t< T} ] \leq C\eps^2.
\eea
Sending $\eps\to 0$ we conclude that
\bea
\mathbb{P}_{x_0}[ X_t \in  S \cap B(x_0,R) \textrm{ for some } { 0<t< T} ] =0\, .
\eea
This proves the the claim.
\end{proof}

Using Claim \ref{claim_key}, and also that both $u$ and $\nabla u$ have a well-behaved heat flow on $(M,d,m)$ thanks to the vanishing capacity of the singular set, the computation from \cite[Section 6.2]{Naber13} (or alternatively the simplified computation from  \cite{HaslhoferNaber_martingales}) now goes through and shows that

\begin{equation}
|\nabla_x \mathbb{E}_x [F]| \leq \mathbb{E}_x [|\nabla^\parallel F|]\, .
\end{equation}
This proves the theorem.
\end{proof}

In the above proof we used the following lemma:

\begin{lemma}\label{std_lemma}
For any geodesic ball $B(x_0,r)$ properly contained in a (possibly incomplete) manifold with non-negative Ricci curvature, denoting by $\tau$ the first exit time, for every $\delta\leq \tfrac{1}{2n}r^2$ we have
\begin{equation}
\mathbb{P}_{x_0}[\tau \leq \delta]\leq e^{-\frac{r^2}{100\delta}}.
\end{equation}
\end{lemma}

\begin{proof}
Write $r(x):=d(x,x_0)$ and consider the radial process $r_t:=r(X_t)$. By a classical result of Kendall \cite{Kendall} this satisfies the evolution equation
\bea
dr_t = d\beta_t +\tfrac{1}{2}\Delta r (X_t)\, dt -dL_t,
\eea
where $\beta_t$ is a one-dimensional Euclidean Brownian motion and $L_t$ is a nondecreasing process (which increases only on the cutlocus). Using Ito calculus we infer that
\bea
dr_t^2 \leq 2 r_t \, d\beta_t + r_t\Delta r(X_t)\, dt + dt,
\eea
where the last term comes from the quadratic variation. Together with the assumption that $\mathrm{Ric}\geq 0$, and hence $\Delta r\leq \frac{n-1}{r}$ by Laplace comparison, this yields
\bea
r_t^2 \leq 2\int_0^t r_s\, d\beta_s +n t\, .
\eea
Taking $t=\tau$ we see that the event $\{\tau\leq \delta\}$ implies
\bea
\int_0^\tau r_s\, d\beta_s\geq \frac{r^2-n\tau}{2}.
\eea
By Levy's criterion there is a Brownian motion $W$ such that
\bea
\int_0^\tau r_s\, d\beta_s =W_{\sigma}, \quad\textrm{ where } \quad
\sigma=\int_0^\tau r_s^2 ds \leq r^2\tau\, .
\eea
Hence $\{\tau\leq \delta\}$ implies
\bea
\max_{0\leq s \leq r^2\delta} W_s \geq \frac{r^2-n\delta}{2}.
\eea
We conclude that
\bea
\mathbb{P}_{x_0}[\tau \leq \delta] \leq \mathbb{P}_{x_0}\left[ |W_1| \geq \frac{1}{r\delta^{1/2}}\frac{r^2-n\delta}{2}\right] \leq e^{-\frac{r^2}{100\delta}}\, .
\eea
This proves the lemma.
\end{proof}

\bigskip

\subsection{Effective intersection estimate for two independent Brownian motions}

In this final section, we prove our effective intersection estimate.

\begin{proof}[{Proof of Theorem \ref{thm_eff_int}}]
Since the two Brownian motions are independent, Theorem \ref{thm-hitting} yields
\bea 
C^{-1} \mathbb{E}_{\tilde{x}_0}[ \mathrm{Cap}_{K_{x_0}}(N_{\eps}(\tilde{X}))]\leq \mathbb{P}_{{x}_0,\tilde x_0}[N_{\eps/2}(X)\cap N_{\eps/2}(\tilde{X}) \neq \emptyset ]
\leq C \mathbb{E}_{\tilde{x}_0}[ \mathrm{Cap}_{K_{x_0}}(N_{\eps}(\tilde{X}))]\, ,
\eea
where we abbreviated $N_\eps(X)=N_\eps(X[0,\infty))$. We thus have the estimate the expected Martin capacity of the Wiener sausage from above and below.\\

Let us first prove the upper bound. To this end, first observe that for any Borel $A\subseteq M$ we can compute the expected volume via the formula
 \bea  \mathbb{E}_{\tilde x_0} [\mathrm{Vol}(N_\e(\tilde X) \cap A )] 
 &= \int_A \mathbb{P}_{\tilde x_0} [\tilde  X_t \in B(y,\e) \text{ for some } 0<t<\infty\, ] \, dV(y) \, .
 \eea
 Using again Theorem \ref{thm-hitting}, as well as the equations \eqref{eqMK} and \eqref{eq_cap_ball}   we can estimate this by 
 \bea\label{eq-37}
 \mathbb{E}_{\tilde x_0} [\mathrm{Vol}(N_\e(\tilde X) \cap A )] 
 &\le C   \int_A\mathrm{Cap}_{K_{\tilde x_0}}(B(y,\e)) \, dV(y) \\
 &\le C (d(\tilde x_0,A)-\e)_+^{2-n} \e^{n-2} \mathrm{Vol} (A)\, .  \eea
 On the other hand, for any Borel set $B\subseteq M$ we can estimate
 \bea
 \mathrm{Cap}_{K_{x_0}}(B)\leq d(x_0,B)^{2-n} \mathrm{Cap}_{N}(B) \leq Cd(x_0,B)^{2-n}\eps^{-2}\textrm{Vol}(N_\eps(B))\, .
 \eea
Combining the above, we infer that for any Borel set $A\subseteq M$ with $d(\tilde x_0,A)\ge 2\e$ we have the useful estimate
 \bea\label{eq_useful} \mathbb{E}_{\tilde x_0} [ \mathrm{Cap}_{K_{x_0}} (N_{\e}(\tilde X) \cap A )] 
 \le  C \e^{n-4} d(x_0,A)^{2-n}d(\tilde x_0,A)^{2-n} \mathrm{Vol}(N_\eps(A)), \eea 
which we will apply repeatedly in the following.
\bigskip

Set  $r_0:= d(x_0,\tilde x_0)$ and let $E:=B(x_0,r_0/2)\cup B(\tilde{x}_0,r_0/2)$. We decompose
\begin{multline}
N_\eps(\tilde X) \setminus E =
N_{\e}(\tilde X) \cap B(x_0,2r_0)\setminus E
\cup\bigcup_{i=1}^\infty N_{\e}(\tilde X) \cap B(x_0,2^{i+1}r_0) \setminus B(x_0,2^ir_0).
\end{multline}
Hence, using \eqref{eq_useful} we can estimate
  \bea \mathbb{E}_{\tilde x_0} [ \mathrm{Cap}_{K_{x_0}} (N_{\e}(\tilde X) \setminus E)]  &\le  \sum_{i=0}^\infty C \e^{n-4} (2^i r_0) ^{2-n}(2^i r_0)^{2-n} (2^i r_0)^n \\
&  \le C \e^{n-4} r_0^{4-n} .\eea 
  It remains to show similar bounds for the expected Martin capacity of $N_\eps(\tilde X) \cap E$. To this end, first observe that Theorem \ref{thm-hitting} yields
 \begin{multline}
    \mathbb{E}_{\tilde x_0} [ \mathrm{Cap}_{K_{x_0}} (N_{\e}(\tilde X) \cap B(x_0,2\e) )] \\
    \leq   \mathrm{Cap}_{K_{x_0}}(B(x_0,2\e)) \mathbb{P}_{\tilde x_0} [ \tilde X_t  \in B(x_0, 3\e) \text{ for some } 0<t<\infty]
\le C \e^{n-2} r_0^{2-n}.
\end{multline}
Now, decomposing $B(x_0,r_0/2)\setminus B(x_0,2\eps)$ into dyadic annuli with radii $2^i \eps$, letting $k$ be the smallest integer such that $2^k\eps\geq r_0/2$, and using \eqref{eq_useful} we can estimate
\bea   \mathbb{E}_{\tilde x_0} [ \mathrm{Cap}_{K_{x_0}} (N_{\e}(\tilde X) \cap B(x_0,r_0/2) \setminus B(x_0,2\e) )] 
   & \le C\sum_{i=1}^k\e^{n-4} r_0^{2-n} (2^i \e)^{2-n } (2^{i}\e)^n \\
   & \le C \e^{n-4} r_0^{4-n}.\eea  
Combining the above estimates, and repeating the same argument with $x_0$ and $\tilde{x}_0$ interchanged, we conclude that
     \bea \mathbb{E}_{\tilde x_0} [ \mathrm{Cap}_{K_{x_0}} (N_{\e}(\tilde X) \cap E)]   \le C \e^{n-4} r_0^{4-n} .\eea 
This finishes the proof of upper bound.

\bigskip 

For the lower bound, it suffices to show $ \mathbb{E}_{\tilde x_0} [\mathrm{Cap}_{K_{x_0}} (N_\e(\tilde X)\cap A )] \ge C^{-1}  {\e}^{n-4}r_0^{4-n}$ for some $A\subseteq M$. With this goal, a convenient choice is  $A:= B( x_0, 4r_0) \setminus B(x_0, 2r_0)$, as there is a positive $C=C(n,v)<\infty$  such that 
\bea\label{eq-comparable}C^{-1} r_0^n \le \mathrm{Vol}(A)\quad \text{ and }\quad C^{-1} r_0 \le  d(x_0,z) , d(\tilde x_0,z) \le C r_0 \quad \text{ for all } z\in A  .\eea
Inserting the uniform measure in the definition of $\mathrm{Cap}_{K_{x_0}}$ we see that
\bea\label{eq-45} \mathrm{Cap}_{K_{x_0}} (N_\e(\tilde X)\cap A )\ge C^{-1}r_0^{2-n}   \frac{ \mathrm{Vol}( N_\e(\tilde X)\cap A)^2 }{ \iint_{(N_\e(\tilde X)\cap A)^2} d(x,y)^{2-n}\, dV(x) dV(y) } .\eea
Arguing similarly as in the derivation of \eqref{eq-37}, but now using the lower bound of Theorem \ref{thm-hitting}, we obtain
\bea\mathbb{E}_{\tilde x_0}[\mathrm{Vol}( N_\e(\tilde X )\cap A)] \ge C^{-1}r_0^{2} \e^{n-2} . \eea 
Together with the Cauchy-Schwarz inequality this yields
\bea\label{eq-46}\mathbb{E}_{\tilde x_0} \left[   \frac{ \mathrm{Vol}( N_\e(\tilde X)\cap A)^2 }{ \iint_{(N_\e(\tilde X)\cap A)^2} d(x,y)^{2-n}\, dV(x) dV(y) }  \right]    \ge  C^{-1}\frac{r_0^{4}\eps^{2n-4}  } {\mathbb{E}_{\tilde x_0} \left[\iint_{(N_\e(\tilde X)\cap A)^2} d(x,y)^{2-n}\, dV(x) dV(y)\right]} . \eea 
To proceed, we rewrite the denominator as
\begin{multline} \label {eq-47} \mathbb{E}_{\tilde x_0} \left[\iint_{(N_\e(\tilde X)\cap A)^2} d(x,y)^{2-n}\, dV(x)dV(y)\right] \\ = \iint_{A\times A}  \mathbb{P}_{\tilde x_0}[\tilde X \text{ hits } B(x,\e) \text{ and } B(y,\e)] d(x,y)^{2-n}\, dV(x) dV(y) .\end{multline}
By symmetry and the strong Markov property we can estimate
\begin{multline}
\iint_{A\times A} \mathbb{P}_{\tilde x_0}[\tilde X \text{ hits } B(x,\e) \text{ and } B(y,\e)]d(x,y)^{2-n}\, dV(x) dV(y)\\\leq 2\iint_{A\times A}  \mathbb{P}_{\tilde x_0}[\tilde X \text{ hits } B(x,\e) ] \sup_{x'\in B(x,\eps)} \mathbb{P}_{x'}[X' \text{ hits } B(y,\e) ]d(x,y)^{2-n}\, dV(x) dV(y) \, .
\end{multline}
Now, using again Theorem \ref{thm-hitting}, for fixed $x$ and $x'$ we have
\begin{multline}
\int_A \mathbb{P}_{x'}[X' \text{ hits } B(y,\e) ] d(x,y)^{2-n} \, dV(y)\\\leq C\int_{A\cap B(x,4\eps)} d(x,y)^{2-n}dV(y) + C\int_{A\setminus B(x,4\eps)} \eps^{n-2} d(x',y)^{2-n} d(x,y)^{2-n} \, dV(y)\leq C\eps^2\, .
\end{multline}
Combining the above formulas we infer that
\bea
 \mathbb{E}_{\tilde x_0} \left[\iint_{(N_\e(\tilde X)\cap A)^2} d(x,y)^{2-n}\, dV(x)dV(y)\right] \leq C\eps^2\int_A \mathbb{P}_{\tilde x_0}[\tilde X \text{ hits } B(x,\e) ] \, dV(x)\leq Cr_0^2\eps^n\, ,
\eea
and thus conclude that
\bea
 \mathbb{E}_{\tilde x_0} [\mathrm{Cap}_{K_{x_0}} (N_\e(\tilde X)\cap A )] \ge C^{-1} r_0^{4-n} {\e}^{n-4}\, .
\eea
This finishes the proof of the theorem.
\end{proof}

 \section{Non-Ricci-flat case}\label{sec-othercase}
 
In this final section, we generalize the hitting estimate from Theorem \ref{thm-hitting} to the case of non-zero Ricci lower bounds. As an application, we obtain an upper bound on the hitting probability of Brownian motion to quantitative singular sets on the manifolds with two-sided Ricci bounds. We will also see that noncollapsed limits of Einstein manifolds with bounded Einstein constants satisfies the same two-sided Ricci bound in the weak sense of Naber \cite{Naber13}. The proofs are relatively minor modifications of the ones from the previous sections. 
 
 \bigskip

Let us start by heat kernel estimates. For complete manifold $(M^n,g)$ with $\mathrm{Ric}\ge -(n-1)g$, we have the heat kernel upper bound 
\bea  \rho_t(x,y) \le C_n \mathrm{Vol}^{-\frac12}(B(x,\sqrt t)) \mathrm{Vol}^{-\frac12}(B(y,\sqrt t)) e^{t - \frac{d(x,y)^2}{8t}}\eea 
 from Li-Yau \cite[Section 3]{LiYau}. On the other hand, Cheeger-Yau \cite{CheegerYau} showed the heat kernel is bounded below  by that on hyperbolic space. The heat kernel on hyperbolic space is analyzed in Davies-Mandouvalos \cite[Theorem 3.1]{DMheat}. Thus we obtain 
 \bea  \rho_t(x,y)\ge {C_n}^{-1}t^{-\frac n2} e^{-\frac{d(x,y)^2}{4t} - \frac{(n-1)^2t}{4} - \frac{(n-1) d(x,y)}{2}}.\eea
 In the remaining part, the following weaker bounds will be sufficient: 	There is a constant $C=C(n,v,T)<\infty$, where $v>0$ is a lower bound for $\mathrm{Vol}(B(x_0,1))$, such that 
 	\bea \label{eq-heatkernelbound2}  \frac{1}{Ct^{n/2}}e^{-\frac{d(x,y)^2}{4t}}\le \rho_t(x,y) \le \frac{C}{t^{n/2}}e^{-\frac{d(x,y)^2}{8t}} \eea
for  $0<t<2T$ 
 
 \bigskip

In the following estimate on hitting estimate, the dependency of $\Lambda$ on $T$ is not as sharp as Theorem \ref{thm-hitting}, mainly because the heat kernel estimates are less sharp than for nonnegative Ricci curvature. However, the estimate is still strong enough for the applications discussed below.
 
\begin{theorem} [c.f. Theorem \ref{thm-hitting}] \label{thm-hitting'}Let $(M^n,g)$ be a complete Riemannian manifold of dimension $n\geq 3$, with $\mathrm{Ric}\ge -(n-1)g$.  Then for all finite $T>0$ and compact sets $A\subseteq M^n$ with $\mathrm{diam}(A\cup \{x_0\})\le R$, we have the hitting estimate
 \bea \frac{1}{\Lambda}  \mathrm{Cap}_{K}(A)\le \mathbb{P}_{x_0}[ X_t \in  A \textrm{ for some } { 0<t< T} ] \le  \Lambda \mathrm{Cap}_{K}(A), \eea 
where $ \Lambda=   \Lambda(n,v,T,R)<\infty$,  with $v = \mathrm{Vol}(B(x_0,1))$. 
\end{theorem}

 \begin{proof}This readily follows from the same argument as in the proof of Theorem \ref{thm-hitting}, except that we now use the bounds \eqref{eq-heatkernelbound2} instead of \eqref{eq_CY} and \eqref{eq_LY}.
\end{proof}
 
 \bigskip
 
 To formulate a theorem for the manifolds with bounded Ricci that is analogous to Theorem \ref{thm_app1} for the Ricci-flat manifold, let us denote the harmonic radius of $x\in M$ by $r_h(x)$ (for the detailed definition see  \cite[Definition 2.9]{CheegerNaber}) and consider quantitative singular set  \bea S_\e ^h:=\{ x\in M\,:\, r_h(x) \le \e\}. \eea  
 They replace the role of the regularity scale $\mathrm{reg}(x)$ in \eqref{eq-regularityscale} and of $S_\e$ in \eqref{eq-esingular}.

  \begin{theorem}[c.f. Theorem \ref{thm_app1}] \label{thm_app1'}Let $(M^n,g)$ be a complete Riemannian manifold of dimension $n\ge 3$ with $-(n-1)g\le \mathrm{Ric}\le (n-1)g$. For all $x_0\in M$ with $d(x_0,S^h_\e)\ge 3\e$, we have the estimate 
 \bea \mathbb{P}_{x_0}\left[N_\eps(X[0,T])\cap N_\eps(S^h_\eps) \cap B(x_0,R) \neq \emptyset  \right]\le \frac{C\e^2}{d(x_0,S^h_\e)^{2}} \eea 
 where $C=C(n,v,R,T)<\infty$ where $v>0$ is a lower bound of $\mathrm{Vol}(B(x_0,1))$. If $(M^n,g)$ is further assumed to be Einstein, then   $S^h_\e$ can be replaced by $S_\e$ in \eqref{eq-esingular} and the same   statement holds.
 \end{theorem} 	
 	
 	\begin{proof}
By Jiang-Naber \cite[Theorem 1.8]{JiangNaber}, for given $R>0$ there exists a constant $C=C(n,v, R)<\infty$ such that $\mathrm{Vol}(N_{2\e} (S^h_\e) \cap B(x_0,2R)) \le C \e^4(2R)^{n-4} $.  Then, by rescaling, there holds  \bea \label{JiangNaber_est'}\mathrm{Vol}(N_{2\e} (S^h_\e) \cap B(x_0,r)) \le C \e^4r^{n-4} \eea for each $0<r<2R$, with the same $C=C(n,v,R)$ as above. Now the proof of Theorem \ref{thm_app1} exactly goes through  once Theorem \ref{thm-hitting} is replaced by Theorem \ref{thm-hitting'} and \eqref{JiangNaber_est} is replaced by \eqref{JiangNaber_est'}.
 
Moreover, if $(M^n,g)$ is Einstein, then standard elliptic estimates imply that the harmonic radius is comparable to the regularity scale, and hence the second assertion follows by adjusting the constant $C$. 
 	\end{proof}

 \begin{definition}[{c.f. Naber  \cite{Naber13}}]\label{def-boundedricci}
A Riemannian metric measure space $(M,d,m)$ satisfies the two-sided Ricci bound $\textrm{BR}(\kappa,\infty)$ if the gradient estimate
\begin{equation}\label{grad_path2'}
|\nabla_x \mathbb{E}_x [F]| \leq \mathbb{E}_x\left[|\nabla^\parallel_0 F|+ \tfrac{\kappa}{2}\int_0^\infty e^{\frac \kappa 2 s}|\nabla^\parallel_s F|\, ds  \right]
\end{equation}
holds for all test functions $F:PM\to\mathbb{R}$ on path space and $m$-almost every $x\in M$, where $\nabla_x$ is the local Lipschitz slope and the (stochastic) parallel gradient $\nabla^\parallel_s$ is the one constructed in \cite[Section 14]{Naber13}.
\end{definition}
 
 \begin{theorem}[c.f. Theorem \ref{thm_ricci_limit}] Every noncollapsed limit of Einstein manifolds $(M_i^n,g_i)$ with $\mathrm{Ric}_{g_i}=\lambda_i g_i$ and $\limsup_{i\to \infty}|\lambda_i|\le  \kappa $ satisfies the  $\textrm{BR}(\kappa,\infty)$ condition.
  \end{theorem}

 	\begin{proof} Let $(M,d,m)$ be the Gromov-Hausdorff limit of the noncollapsed Einstein manifolds, and let $S\subset M$ be the singular set.
	Using Theorem \ref{thm_app1'} instead of Theorem \ref{thm_app1}, and arguing similarly as in the proof of Claim \ref{claim_key} we see that Brownian motion stays in $M\setminus S$ almost surely. Observing also that on $M\setminus S$ we have $-\kappa  \leq \mathrm{Rc} \leq \kappa$, the rest of the argument is similar as before. 
 	\end{proof}

\bibliography{ChoiHaslhofer}

\bibliographystyle{alpha}

\vspace{10mm}

{\sc Beomjun Choi, Department of Mathematics, University of Toronto,  40 St George Street, Toronto, ON M5S 2E4, Canada}\\

{\sc Robert Haslhofer, Department of Mathematics, University of Toronto,  40 St George Street, Toronto, ON M5S 2E4, Canada}\\

\end{document}